\documentclass{amsart}
\usepackage{amssymb,amsmath,latexsym}
\usepackage{amsthm}
\usepackage{fontenc}
\usepackage{amssymb}

\numberwithin{equation}{section}

\newtheorem{theorem}{Theorem}[section]

\begin{document}
\author{Alexander E Patkowski}
\title{A Note on the Axisymmetric Diffusion equation}

\maketitle
\begin{abstract} We consider the explicit solution to the axisymmetric diffusion equation. We recast the solution in the form of a Mellin inversion formula, and outline a method to compute a formula for $u(r,t)$ as a series using the Cauchy residue theorem. As a consequence, we are able to represent the solution to the axisymmetric diffusion equation as rapidly converging series.\end{abstract}

\keywords{\it Keywords: \rm Axisymmetric Diffusion equation; Bessel functions; Mellin transforms}

\subjclass{ \it 2010 Mathematics Subject Classification 35B40, 35C15, 35C10}

\section{Introduction and Main results} 

The axisymmetric diffusion equation is [2, pg.61]
\begin{equation}\kappa \nabla^2u\equiv\kappa\frac{1}{r}\frac{\partial}{\partial r}\left(r \frac{\partial u}{\partial r}\right)=\kappa\left(u_{rr}+\frac{1}{r}u_r\right)=\frac{\partial u}{\partial t},\end{equation}
where $t>0,$ $r\in(0,\infty),$ $u(r,0)=g(r),$ and positive diffusitivity constant $\kappa.$ The Hankel transform of a function $f(x)$ is defined as [2, pg.58, eq.(1.10.1)]
$$\mathfrak{H}(f(y))(x):=\int_{0}^{\infty}yJ_0(xy)f(y)dy.$$ We may temporarily drop the integrating variable in denoting integral transforms according to when the context is appropriate throughout. The known explicit solution is obtained by taking Hankel transform of (1.1), which gives
\begin{equation}\frac{\partial}{\partial t}\mathfrak{H}(u(r,t))(x)+x^2\kappa\mathfrak{H}(u(r,t))(x)=0,\end{equation}
with initial condition $\mathfrak{H}(u(r,0))(x)=\mathfrak{H}(g(r)).$ Applying the inverse Hankel transform $\mathfrak{H}^{-1}$ to (1.2) gives the explicit solution [2, pg.62, eq.(1.10.25)]
\begin{equation}u(r,t)=\frac{1}{2\kappa t}e^{-r^2/(4\kappa t)}\int_{0}^{\infty}yg(y)I_0(\frac{yr}{2\kappa t})e^{-y^2/(4\kappa t)}dy,\end{equation}
where the modified Bessel function of the first kind is given by
$$I_v(x)=\sum_{n\ge0}\frac{1}{n!\Gamma(v+n+1)}\left(\frac{x}{2}\right)^{2n+v}.$$
Some simple examples include the bell-shaped temperature profile $g(r)=e^{-cr^2},$ or the uniform temperature profile $g(r)=1$ on $(0,1).$ In both these instances it is a simple task to appeal to the tables. 

The purpose of this note is to provide further analysis on (1.3) by means of Mellin inversion. In applying methods from [4] we can better understand $u(r,t)$ by providing a method to obtain an infinite series representation involving Laguerre polynomials or a hypergeometric function. For a general overview applying Mellin transforms to evaluating integrals involving Bessel functions see [6, pg.196]. \par Recall [4] the Mellin transform is given by 
$$\mathfrak{M}(g)(s):=\int_{0}^{\infty}y^{s-1}g(y)dy.$$ Parseval's identity is [4, pg.83, eq.(3.1.11)]
\begin{equation}\int_{0}^{\infty}k(y)g(y)dy=\frac{1}{2\pi i}\int_{(c)}\mathfrak{M}(k)(s)\mathfrak{M}(g)(1-s)ds \end{equation}

Recall from [3, pg.709, eq.(6.643), $\#2$] (with change of variable $x\rightarrow x^2$ and $\mu=\frac{s}{2}$) that
\begin{equation}\int_{0}^{\infty}y^{s}e^{-\alpha y^2}I_{2v}(2\beta y)dy=\frac{\Gamma(\frac{s}{2}+v+\frac{1}{2})e^{\beta^2/(2\alpha)}}{2\Gamma(2v+1)\beta}\alpha^{-s/2}M_{-s/2,v}(\frac{\beta^2}{\alpha}),\end{equation}
valid for $\Re(\frac{s}{2}+v+\frac{1}{2})>0.$

Here $M_{\mu,v}(x)$ is the Whittaker hypergeometric function [3, pg.1024]
\begin{equation}M_{\mu,v}(x)=x^{v+\frac{1}{2}}e^{-x/2}{}_1F_1(v-\mu+\frac{1}{2};2v+1;x),\end{equation} and $_1F_1(a;b;x)$ is the confluent hypergeometric function.

\begin{theorem} If $\mathfrak{M}(g)(1-s)$ is analytic in a subset $S$ of the region $\{s\in\mathbb{C}:\Re(s)>-1\},$ then
$$u(x,t)=\frac{1}{r}e^{-(r^2-\frac{r^2}{2})/(4\kappa t)}\frac{1}{2\pi i}\int_{(c)}\Gamma(\frac{s}{2}+\frac{1}{2}) (4\kappa t)^{s/2}M_{-s/2,0}(\frac{r^2}{4\kappa t})\mathfrak{M}(g)(1-s)ds,$$
$c\in S\cap\{s\in\mathbb{C}:\Re(s)>-1\}.$
\end{theorem}

\begin{proof}
We choose the $k(y)$ to be the integrand in (1.5) with $v=0,$ $\alpha=\frac{1}{4\kappa t},$ $\beta=\frac{r}{4\kappa t},$ and apply (1.4).
\end{proof}
Some relevant notes are in order to apply Theorem 1.1. First, Theorem 1.1 requires that $\mathfrak{M}(g)(s)$ is analytic in the region $\{s\in\mathbb{C}:\Re(s)<2\}.$ It is known that $M_{\mu,v}(x)$ only has simple poles for fixed $\mu$ and $x,$ at $v=-\frac{1}{2}(k+1),$ $k\in\mathbb{N}.$ By [3, pg.1026, 9.228], 
$$M_{\mu,v}(x)\sim \frac{1}{\sqrt{\pi}}\Gamma(2v+1)\mu^{-v-\frac{1}{2}}x^{1/4}\cos(2\sqrt{\mu x}-v\pi-\frac{\pi}{4}),$$
as $|\mu|\rightarrow\infty,$ and further we have the functional relationship [3, pg.1026, eq.(9.231, $\#2$],
$$x^{-\frac{1}{2}-v}M_{\mu,v}(x)=(-x)^{-\frac{1}{2}-v}M_{-\mu,v}(-x).$$
In the model with $u(r,0)=J_0(ar),$ the Bessel function of the first kind, we may proceed in the following way. Note that for $-v<\Re(s)<\frac{3}{2},$ [4, pg. 407]
\begin{equation}\mathfrak{M}(J_v(ay))(s)=\frac{2^{s-1}\Gamma(\frac{v}{2}+\frac{s}{2})}{\Gamma(1+\frac{v}{2}-\frac{s}{2})}a^{-s}.\end{equation}
We set $v=0$ and insert (1.7) into Theorem 1.1 to obtain for $-\frac{1}{2}<\Re(s)=c<1,$
\begin{equation}u(r,t)=\frac{1}{r}e^{-(r^2-\frac{r^2}{2})/(4\kappa t)}\frac{1}{2\pi i}\int_{(c)}(4\kappa t)^{s/2}M_{-s/2,0}(\frac{r^2}{4\kappa t})2^{-s}\Gamma(\frac{1-s}{2})a^{s-1}ds.\end{equation}
It is known that 
\begin{equation} {}_1F_1(a,1;x)=e^{x}L_{a-1}(-x),\end{equation} where $L_n(x)$ is the Laguerre polynomial [3]. This can be seen by using [3, pg.1001] $L_a(x)={}_1F_1(-a;1;x),$ together with Kummer's [1, pg.509] ${}_1F_1(a;b;x)=e^{x}{}_1F_1(1-a;b;-x)$ with $b=1.$ Now (1.6) with (1.8) leads to 
\begin{equation}\begin{aligned}u(r,t)&=\frac{e^{-r^2/(4\kappa t)}}{\sqrt{4\kappa t}}\frac{1}{2\pi i}\int_{(c)}(4\kappa t)^{s/2}{}_1F_1(\frac{s}{2}+\frac{1}{2};1;\frac{r^2}{4\kappa t})2^{-s}\Gamma(\frac{1-s}{2})a^{s-1}ds\\
&=\frac{e^{-r^2/(4\kappa t)}}{\sqrt{4\kappa t}}\frac{1}{2\pi i}\int_{(1-c)}(4\kappa t)^{(1-s)/2}{}_1F_1(1-\frac{s}{2};1;\frac{r^2}{4\kappa t})2^{s-1}\Gamma(\frac{s}{2})a^{-s}ds.\end{aligned}\end{equation}
Here we made the change of variable $s\rightarrow 1-s.$ This integral has simple poles at $s=0,$ and the negative even integers $s=-2n.$ Computing these residues and using (1.9) gives
$$u(r,t)=\sum_{n\ge0}\frac{L_n(-\frac{r^2}{4\kappa t})}{n!}\left(-a^2\kappa t\right)^n=e^{-a^2\kappa t}\frac{1}{2}J_0(ar).$$
Here we have applied the $\alpha=0$ case of [5, pg.102, Theorem 5.1, eq.(5.1.16)]
$$\sum_{n\ge0}\frac{L^{(\alpha)}_n(x)}{\Gamma(n+\alpha+1)}w^n=e^{w}(xw)^{-\alpha/2}J_{\alpha}(2\sqrt{xw}).$$
Next we consider an example of Theorem 1.1 with a function for which it is difficult to evaluate (1.3), and is apparently new.
\begin{theorem} The solution of (1.1) with $u(r,0)=J_0^2(ar),$ is given by 
$$u(r,t)= \frac{1}{2}\sum_{n\ge0}\frac{(2n)!}{(n!)^3}L_{n}(-\frac{r^2}{4^2\kappa t})(-a^2\kappa t)^{n}.$$ 
\end{theorem}

\begin{proof} First we write down [4, pg.407]
\begin{equation}\mathfrak{M}(J_v^2(ay))(s)=\frac{2^{s-1}\Gamma(\frac{s}{2}+v)\Gamma(1-s)}{\Gamma^2(1-\frac{s}{2})\Gamma(1+v-\frac{s}{2})}a^{-s}.\end{equation}
valid for $-\Re(v)<\Re(s)<1.$ We set $v=0$ in (1.11), and insert it into Theorem 1.1 to find for $0<c<1,$
$$\begin{aligned}u(r,t)&=\frac{1}{r}e^{-r^2/(4\kappa t)}\frac{1}{2\pi i}\int_{(c)}(4\kappa t)^{s/2}M_{-s/2,0}(\frac{r^2}{4\kappa t})\frac{2^{-s}\Gamma(\frac{1-s}{2})\Gamma(s)}{\Gamma^2(\frac{1}{2}+\frac{s}{2})}a^{s-1}ds \\
&=\frac{e^{-r^2/(4\kappa t)}}{\sqrt{4\kappa t}}\frac{1}{2\pi i}\int_{(1-c)}(4\kappa t)^{(1-s)/2}{}_1F_1(1-\frac{s}{2};1;\frac{r^2}{4\kappa t})\frac{2^{s-1}\Gamma(\frac{s}{2})\Gamma(1-s)}{\Gamma^2(1-\frac{s}{2})}a^{-s}ds 
.\end{aligned}$$
The resulting integral has simple poles $s=-2n$ for each integer $n\ge0.$ Therefore, computing the residues at these poles gives, by Cauchy's residue theorem and (1.9),
$$u(r,t)=\frac{1}{2}\sum_{n\ge0}\frac{(2n)!}{(n!)^3}L_{n}(-\frac{r^2}{4^2\kappa t})(-a^2\kappa t)^{n}.$$ \end{proof}
\par It is interesting to note that taking the limit $r\rightarrow0$ of Theorem 1.2 gives
$$\lim_{r\rightarrow0}\left(\sum_{n\ge0}\frac{(2n)!}{(n!)^3}L_n(-\frac{r^2}{4^2\kappa t})(-a^2\kappa t)^n\right)=e^{-a^2\kappa t}I_0(a^2\kappa t),$$
by means of [3, pg.1024, eq.(9.215), $\#3,$ $p=0,$ $z=ix$].
Next we consider an initial condition involving the modified Bessel function of the second kind $K_v(x),$ which has the general relationship [3]
$$K_v(x)=\frac{\pi(I_{-v}(x)-I_v(x))}{2\sin(\pi v)}.$$ 
\begin{theorem} The solution to (1.1), with $u(r,0)=I_v(ar)K_v(ar),$ is given by
$$u(r,t)=e^{-r^2/(4\kappa t)}\frac{(4\kappa t a^2)^{v}}{4\sqrt{\pi}}\sum_{n\ge0}{}_1F_1(1+v+n;1;\frac{r^2}{4\kappa t})\frac{\Gamma(1+v+n)\Gamma(-v-n)\Gamma(\frac{1}{2}+n+v)}{n!\Gamma(2v+1+n)}(-a^24\kappa t)^n$$
$$+\frac{e^{-r^2/(4\kappa t)}}{4\sqrt{\pi}}\sum_{n\ge0}{}_1F_1(1+n;1;\frac{r^2}{4\kappa t})\frac{\Gamma(v-n)\Gamma(\frac{1}{2}+n)}{\Gamma(v+1+n)}(-a^24\kappa t)^n.$$
provided that $v$ is not an integer or equal to $0.$
\end{theorem}
\begin{proof} From [6, pg.199, eq.(7.10.8)] with $0<\Re(s)<1,$
\begin{equation}\mathfrak{M}(I_v(ay)K_v(ay))(s)=\frac{\Gamma(\frac{s}{2}+v)\Gamma(\frac{1}{2}-\frac{s}{2})\Gamma(\frac{s}{2})}{4\sqrt{\pi}\Gamma(v+1-\frac{s}{2})}a^{-s}.\end{equation}
Setting $v=0$ in (1.12) and applying Theorem 1.1, we have that $u(r,t)$ is equal to
$$\begin{aligned}&\frac{1}{r4\sqrt{\pi}}e^{-r^2/(4\kappa t)}\frac{1}{2\pi i}\int_{(c)}(4\kappa t)^{s/2}\Gamma(\frac{1}{2}+\frac{s}{2})M_{-s/2,0}(\frac{r^2}{4\kappa t})\frac{\Gamma(\frac{1}{2}(1-s)+v)\Gamma(\frac{1}{2}-\frac{s}{2})\Gamma(\frac{s}{2})}{\Gamma(v+\frac{1}{2}+\frac{s}{2})}a^{s-1}ds \\
&=\frac{e^{-r^2/(4\kappa t)}}{4\sqrt{\pi}\sqrt{4\kappa t}}\frac{1}{2\pi i}\int_{(1-c)}(4\kappa t)^{(1-s)/2}{}_1F_1(1-\frac{s}{2};1;\frac{r^2}{4\kappa t})\frac{\Gamma(1-\frac{s}{2})\Gamma(\frac{s}{2}+v)\Gamma(\frac{s}{2})\Gamma(\frac{1-s}{2})}{\Gamma(v+1-\frac{s}{2})}a^{-s}ds 
.
\end{aligned}$$
Now we see that if $v=0$ then the gamma functions would have a pole of order two at the negative even integers $s=-2n,$ which we want to avoid due to the lengthy resulting formula. Hence we restrict $v$ to be a non-integer and $v\neq0,$ and the poles at $s=-2n-2v,$ and $s=-2n$ are simple. For the poles at $s=-2n-2v,$ we have the residue
$$e^{-r^2/(4\kappa t)}\frac{(4\kappa t a^2)^{v}}{4\sqrt{\pi}}\sum_{n\ge0}{}_1F_1(1+v+n;1;\frac{r^2}{4\kappa t})\frac{\Gamma(1+v+n)\Gamma(-v-n)\Gamma(\frac{1}{2}+n+v)}{n!\Gamma(2v+1+n)}(-a^24\kappa t)^n,$$
and for the poles at $s=-2n,$ we have the residue
$$\frac{e^{-r^2/(4\kappa t)}}{4\sqrt{\pi}}\sum_{n\ge0}{}_1F_1(1+n;1;\frac{r^2}{4\kappa t})\frac{\Gamma(v-n)\Gamma(\frac{1}{2}+n)}{\Gamma(v+1+n)}(-a^24\kappa t)^n.$$\end{proof}
A nice consequence of our series representations of $u(r,t)$ is that they are rapidly converging, and so should be of great interest for numerical calculations. From [ pg.1003, eq.(8.978), $\#3,$ $\alpha=0$], we have the asymptotic expansion for the Laguerre polynomial 
\begin{equation} L_n(x)=\frac{e^{x/2}}{\sqrt{\pi}}(xn)^{-1/4}\cos(2\sqrt{nx}-\frac{\pi}{4})+O(n^{-3/4}),\end{equation} as $n\rightarrow\infty,$ uniformly in $x>0.$ In conjunction with our series involving Laguerre polynomials, (1.13) may be used to obtain approximations to $u(r,t).$ 

\section{Some related observations}
We mention a method of evaluating (1.3) when $g(y)=h(y)\log(y)$ for a suitable function $h(y).$ It is known [3, pg.919, eq.(8.447] that
\begin{equation} I_0(x)\log(\frac{x}{2})=-K_0(x)+\sum_{n\ge1}\frac{x^{2n}}{2^{2n}(n!)^2}\psi(n+1),\end{equation}
where $\psi(x)$ is the digamma function [3]. The formula (2.1) appears to provide an effective way of computing special cases of (1.3). We provide an outline of a method.
\begin{theorem} Let $h(y)$ be a suitable function chosen so the series converges. The solution to (1.1) with initial condition $u(r,0)=h(r)\log(r),$ satisfies
$$u(r,t)=\frac{1}{2\kappa t}e^{-r^2/(4\kappa t)}\bigg(\log(\frac{4\kappa t}{r})\mathfrak{Z}_{1}(h)$$
$$-\int_{0}^{\infty} yh(y)e^{-y^2/(4\kappa t)}K_0(\frac{yr}{2\kappa t})dy+\sum_{n\ge1}\frac{\psi(n+1)}{2^{2n}(n!)^2}\left(\frac{r}{2\kappa t}\right)^{2n}\mathfrak{Z}_{2n+1}(h)\bigg),$$
where
$$\mathfrak{Z}_{s}(h):=\mathfrak{M}(yh(y)e^{-y^2/(4\kappa t)})(s)=\int_{0}^{\infty}h(y) y^{s} e^{-y^2/(4\kappa t)}dy.$$
\end{theorem}
\begin{proof}Note that (2.1) implies
$$I_0(\frac{yr}{2\kappa t})\log(y) =\log(\frac{4\kappa t}{r})-K_0(\frac{yr}{2\kappa t})+\sum_{n\ge1}\frac{\psi(k+1)}{2^{2k}(k!)^2}\left(\frac{yr}{2\kappa t}\right)^{2k}.$$
Hence,
$$\int_{0}^{\infty} yh(y)I_0(\frac{yr}{2\kappa t})\log(y)e^{-y^2/(4\kappa t)}dy=\log(\frac{4\kappa t}{r})\int_{0}^{\infty} yh(y)e^{-y^2/(4\kappa t)}dy$$
$$-\int_{0}^{\infty} yh(y)e^{-y^2/(4\kappa t)}K_0(\frac{yr}{2\kappa t})dy+\sum_{n\ge1}\frac{\psi(k+1)}{2^{2k}(k!)^2}\left(\frac{r}{2\kappa t}\right)^{2k}\mathfrak{Z}_{2k}(h(y)),$$
provided $yh(y)\log(y)$ satisfies certain growth conditions. In particular by [9, pg.920] $K_{0}(t)=O(e^{-t}/\sqrt{t}),$ when $t\rightarrow\infty$ in $|\arg(t)|<\frac{3\pi}{2},$ and so we require the very mild necessary (but not sufficient) condition that for a positive constant $c_1,$ and any $t>0,$
$$\big|yh(y)\big|<c_1e^{y^2/(4\kappa t)},$$
by the first integral on the right side.  
\end{proof}
In closing we mention it is possible to utilize many of our results outside of their range of convergence as asymptotic formulas.

1390 Bumps River Rd. \\*
Centerville, MA
02632 \\*
USA \\*
E-mail: alexpatk@hotmail.com, alexepatkowski@gmail.com

\end{document}